\documentclass[a4paper,14pt]{article} 
\usepackage[T2A]{fontenc}
\usepackage[utf8]{inputenc}
\usepackage[russian,english]{babel} 
\usepackage{amssymb,amsfonts,amsmath,mathtext,enumerate,float,amsthm} 
\usepackage[unicode,colorlinks=true,citecolor=black,linkcolor=black]{hyperref}
\usepackage{indentfirst} 
\usepackage[dvips]{graphicx} 
\graphicspath{{illustr/}}

\makeatletter
\renewcommand{\@biblabel}[1]{#1.} 
\makeatother 

\usepackage{geometry} 
\usepackage[backend=biber,style=gost-numeric,sorting=none]{biblatex}
\theoremstyle{plain}
\newtheorem{theorem}{Theorem}[section]
\newtheorem{conjecture}[theorem]{Conjecture}
\newtheorem{lemma}[theorem]{Lemma}
\newtheorem{corollary}[theorem]{Corollary}
\newtheorem{proposition}[theorem]{Proposition}

\theoremstyle{definition}
\newtheorem{definition}[theorem]{Definition}
\newtheorem{construction}[theorem]{Construction}
\newtheorem{remark}[theorem]{Remark}
\newtheorem{problem}[theorem]{Problem}

\begin{document}


\title{
	On existence of integral point sets and their diameter bounds
	\footnote{
		This work was carried out at Voronezh State University and supported by the Russian Science
		Foundation grant 19-11-00197.
	}
}

\author{
	N.N. Avdeev
	\footnote{nickkolok@mail.ru, avdeev@math.vsu.ru}
}

\maketitle

\paragraph{Abstract.}
A point set $M$ in $m$-dimensional Euclidean space is called an integral point set if all the distances between the
elements of $M$ are integers, and $M$ is not situated on an $(m-1)$-dimensional hyperplane.
We improve the linear lower bound for diameter of planar integral point sets.
This improvement takes into account some results related to the Point Packing in a Square problem.
Then for arbitrary integers $m \geq 2$, $n \geq m+1$, $d \geq 1$
we give a construction of an integral point set $M$ of $n$ points in $m$-dimensional Euclidean space,
where $M$ contains points $M_1$ and $M_2$ such that distance between $M_1$ and $M_2$ is exactly $d$.

\section{Introduction}
Let $\mathbb{N}$ be the set of all positive integers and let $|M_1 M_2|$ denote the Euclidean distance
between points $M_1$ and $M_2$ in a finite-dimensional space $\mathbb{R}^m$
(and, more generally, let $|\Delta|$ denote the length of line segment $\Delta$).
An \textit{integral point set} in $m$-dimensional Euclidean space is a point set $M$ such that all the distances between the
points of $M$ are integers and $M$ is not situated on an $(m-1)$-dimensional hyperplane.
Erd\"os and Anning proved~\cite{anning1945integral,erdos1945integral} that every integral point set consists of a finite number of points.
Taking this into account, we denote the set of all integral point sets of $n$ points in $m$-dimensional Euclidean space by
$\mathfrak{M}(m,n)$ (using the notation in~\cite{our-vmmsh-2018})
and denote the set of all integral point sets in $m$-dimensional Euclidean space by $\mathfrak{M}(m,\mathbb{N})$.
The symbol $\# M$ will be used for cardinality of $M$, that is the number of points in $M$ in our case.

For every finite point set, its diameter is naturally defined as
\begin{equation}
	\operatorname{diam} M = \max_{A,B\in M} |AB|
	.
\end{equation}
Another emerging question is: how does the diameter of an integral point set depend on its cardinality?
One can easily see that every $M\in\mathfrak{M}(m,n)$ with $\operatorname{diam} M = h$
can be dilated to $M_p\in\mathfrak{M}(m,n)$ with $\operatorname{diam} M = ph$
for every $p\in\mathbb{N}$.
So, the above question should be rephrased:
how does \textit{the least possible} diameter of an integral point set depend on its cardinality?
In order to answer this question, the following function was introduced~\cite{kurz2008bounds,kurz2008minimum}:
\begin{equation}
	d(m,n) = \min_{M\in\mathfrak{M}(m,n)} \operatorname{diam} M = \min_{M\in\mathfrak{M}(m,n)} \max_{A,B\in M} |AB|
	.
\end{equation}
We also refer to~\cite{kurz2008bounds} for a list of known exact values of $d(m,n)$
and its bounds; in the present paper, the case of $m=2$ will mostly be in the focus.

The most significant breakthrough on the planar case was done by Solymosi~\cite{solymosi2003note},
who proved that $cn \leq d(2,n)$ for a sufficiently small constant $c$.
Following Solymosi's proof carefully,
one can derive that the inequality holds at least for $c = 1/24$.
(See~\cite[Exercise 2.6]{garibaldi2005erdos} for some remarks.)
The constant was improved in~\cite{our-mz-rus} to $1/8$ for all $n$ and in~\cite{our-vmmsh-2018}
to $3/8$ for sufficiently large $n$.

The paper~\cite{solymosi2003note} contains one more interesting result.
Let us define a function which is ``dual'' to $d(m,n)$ in some sense:
\begin{equation}
	l(m,n) = \min_{M\in\mathfrak{M}(m,n)} \min_{A,B\in M} |AB|
	.
\end{equation}
Solymosi proved that $l(2,n)\leq 2$.

In the present paper we improve Solymosi's results:
first, we obtain a larger constant $c = 5/11$ in Theorem~\ref{thm:main_estimate},
using the combined approach with the Point Packing in a Square problem
(this approach is different from Solymosi's one);
second, we prove that $l(m,n)=1$ for all possible $m$ and $n$.

\section{Lower bound for the diameter}

In this section,
we improve the lower bound for minimum diameter of planar integral point sets
employing Point Packing in a Square problem.
Below we introduce the problem, basic notions and results.

\begin{problem}[Point Packing in a Square (PPS)~\cite{locatelli2002packing,costa2013valid}]
	\label{problem:PPS}
	Given an integer $k > 1$, place $k$ points in the unit square $U = [ 0 , 1 ]^2$ such that their
	minimum pairwise distance $m$ is maximal.
\end{problem}

\begin{definition}
	For each $k > 1$, the corresponding maximal distance $m$ from Problem~\ref{problem:PPS}
	is called the $k$-th PPS coefficient and denoted by $\varphi_k$.
\end{definition}
So, it's impossible to place $k$ points in a unit square in such a way that each pairwise distance of the points is greater than $\varphi_k$.

\begin{theorem}
	\label{thm:varphi_k_bounds}
	\cite{costa2013valid}
	For every $k\geq 2$ the following inequality holds:
	\begin{equation*}
		\sqrt{\frac{2}{k\sqrt{3}}}
		\leq
		\varphi_k
		\leq
		\frac{1}{k-1} +
		\sqrt{
			\frac{1}{(k-1)^2}
			+
			\frac{2}{(k-1)\sqrt{3}}
		}
	\end{equation*}
\end{theorem}

To prove the bound on $d(2,n)$, we also need the following results and notions from~\cite{our-vmmsh-2018}.

\begin{lemma}
	\cite[Lemma 4]{our-vmmsh-2018}
	\label{lem:square_container}
	Let $M\in\mathfrak{M}(2,n)$, $\operatorname{diam} M = d$.
	Then $M$ is situated in a square of side length $d$.
\end{lemma}

\begin{definition}
	A \textit{cross} for points $M_1$ and $M_2$, denoted by $cr(M_1,M_2)$, is the union of two straight lines:
	the line through $M_1$ and $M_2$,
	and the perpendicular bisector of line segment $M_1 M_2$.
\end{definition}

\begin{lemma}
	\label{lem:intervals_cross}
	If open line segments $M_1 M_2$ and $M_3 M_4$ do not intersect,
	then the set $cr(M_1,M_2) \cap cr(M_3,M_4)$ is either a straight line or contains 2 or 4 points.
\end{lemma}
It is important that we consider the intersection of open line segments,
so e.g. the cases $M_1 = M_3$ and $M_3 \in M_1 M_2$ satisfy the conditions of Lemma~\ref{lem:intervals_cross}.

\begin{lemma}
	\label{lemma:quadr_diag_edges}
	Let $ABCD$ be a convex quadrilateral on the plane.
	Then $\max\{AC,BD\}>\min\{AB,BC,CD,DA\}$,
	that is at least one diagonal is greater than at least one side.
\end{lemma}

Basing on the exact values of $d(2,n)$ for $ 3 \leq n\leq 122$ and the estimate $d(2,123)>10000$
\cite{kurz2008bounds}, we derive the following proposition.
\begin{proposition}
	\label{obs:4_leq_n_leq_21491}
	The inequality
	\begin{equation}
		d(2,n) \geq 3^{1/4}\cdot2^{-3/2} \cdot n
	\end{equation}
	holds for $4 \leq n \leq 21491$.
\end{proposition}

Therefore, we will focus on planar integral points sets of more than 21491 points.

Performing some simple manipulations with the upper bound in Theorem~\ref{thm:varphi_k_bounds},
one can derive the following proposition.
\begin{proposition}
	\label{varphi_n_where_n_geq_21492}
	For $n \geq 21492$ we have
	\begin{equation}
		\varphi_n \leq \varphi_{n-1} \leq \frac{\beta}{\sqrt{n-2}}
		,
	\end{equation}
	where
	\begin{equation}
		\beta = \frac{1}{\sqrt{21490}} + \sqrt{ \frac{2}{\sqrt{3}} + \frac{1}{21490} } < 1.07464
		.
	\end{equation}
\end{proposition}

Now we need to estimate the cardinality of an intersection of an integral point set with a line segment.
Assuming that the planar integral point sets contains many collinear points,
the following result holds.
\begin{theorem}~\cite[Theorem 4]{kurz2008minimum}
	For $\delta > 0$, $\varepsilon > 0$, and $P\in\mathfrak{M}(2,n)$ with
	at least $n^\delta$ collinear points there exists a $n_0 (\varepsilon)$
	such that for all $n \geq n_0 (\varepsilon)$ we have
	\begin{equation}
		\operatorname{diam} P \geq n^{\frac{\delta}{4 \log 2(1+\varepsilon)}\log \log n}
		.
	\end{equation}
\end{theorem}
However, the estimate is rather unsuitable for our needs,
as it does not provide the values of all the constants.
To obtain the needed estimate, we now prove a generalization of~\cite[lemma 3]{our-vmmsh-2018}.

\begin{definition}
	For a line segment $M_1 M_2$ and an integer $k$, such that $-|M_1 M_2| < k < |M_1 M_2|$,
	we define a $\rho(k,M_1 M_2)$-curve as the set of points $N$
	for which the equality $|N M_1| - |N M_2| = k$ holds.
\end{definition}
So, in the planar case a $\rho(k,M_1 M_2)$-curve
is a branch of a hyperbola for $k\neq 0$ and the perpendicular bisector of line segment $M_1 M_2$ for $k=0$.

\begin{proposition}
	\label{obs:rho_curves}
	If points $M_1,M_2,M_3,M_4$ are situated on a straight line,
	the equality $|M_1 M_2| = |M_3 M_4|$ holds and line segments $M_1 M_2$ and $M_3 M_4$ do not coincide,
	then for a fixed $k$ the $\rho(k,M_1 M_2)$-curve and the $\rho(k,M_3 M_4)$-curve
	do not intersect.
\end{proposition}

\begin{lemma}
	\label{lem:2k-1_segments}
	Let $M \in \mathfrak{M}(2,\mathbb{N})$ and let $m$ be a straight line.
	Then for every $k\in\mathbb{N}$ there are at most $2k-1$ segments $\Delta_i \subset m$ with endpoints in $M$,
	such that $|\Delta_i| = k$.
\end{lemma}
\begin{proof}
	Consider a point $N\in M \setminus m$.
	Then for each $\Delta_i$ there is a $\rho(n_i,\Delta_i)$-curve containing $N$.
	Due to Proposition~\ref{obs:rho_curves}, all $n_i$ are distinct;
	otherwise the $\rho(n_i,\Delta_i)$-curve and the $\rho(n_i,\Delta_j)$-curve, $j\neq i$, do not intersect.
	There can be only $2k-1$ distinct values for $n_i$,
	so there are at most $2k-1$ distinct segments $\Delta_i$.
\end{proof}

\begin{lemma}
	\label{lem:line_segment_with n_squared_plus_one_points}
	Let $\Delta$ be a straight line segment, $|\Delta|=l$ and $M \in \mathfrak{M}(2,\mathbb{N})$.
	Let $\#(\Delta \cap M) = n^2 + 1$.
	Then
	\begin{equation}
		l \geq \frac{2}{3}n^3+\frac{1}{2}n^2-\frac{1}{6}n
		.
	\end{equation}
\end{lemma}

\begin{proof}
	Any $n^2+1$ points, including the endpoints of $\Delta$, partition the segment $\Delta$ into $n^2$
	sequential segments $\Delta_i$.
	Due to Lemma~\ref{lem:2k-1_segments}, there is at most one segment of length 1,
	at most three segments of length 2, etc.
	The following two expressions for sums conclude the proof:
	\begin{equation}
		1 + \sum_{k=1}^n (2k-1) = n^2 + 1
		,
	\end{equation}
	\begin{equation}
		\sum_{k=1}^n k(2k-1) = \frac{2}{3}n^3+\frac{1}{2}n^2-\frac{1}{6}n
		.
	\end{equation}
\end{proof}

Now we will estimate the length of a line segment that intersects an integral point set
by an arbitrary number of points.

\begin{lemma}
	Let $\Delta$ be a straight line segment, $|\Delta| = b$ and $M \in \mathfrak{M}(2,\mathbb{N})$.
	Let $\#(\Delta \cap M) = t$.
	Then
	\begin{equation}
		\label{eq:estimate_for_segment_length}
		b\geq \frac{2}{3}t^{3/2}-\frac{3}{2}t+\frac{5}{6}t^{1/2}
		.
	\end{equation}
\end{lemma}

\begin{proof}
	Let $f(k)$ denote the mininal length of a line segment
	that intersects an integral point set by $k$ points.
	We observe that $f(k) > f (k-1)$.
	Due to Lemma~\ref{lem:line_segment_with n_squared_plus_one_points},
	$f(n^2+1) \geq \frac{2}{3}n^3+\frac{1}{2}n^2-\frac{1}{6}n$.

	For $t\in\mathbb{N}$ the inequality $(\sqrt{t} - 1)^2 +1 \leq t$ holds,
	thus
	\begin{equation}
		\label{eq:line_segment_length}
		f(t) \geq f((\sqrt{t} - 1)^2 +1) \geq \frac{2}{3}(\sqrt{t} - 1)^3+\frac{1}{2}(\sqrt{t} - 1)^2-\frac{1}{6}(\sqrt{t} - 1)
		=
		\frac{2}{3}t^{3/2}-\frac{3}{2}t+\frac{5 }{6}t^{1/2}
		.
	\end{equation}
\end{proof}

This lemma leads to the following proposition.

\begin{proposition}
	\label{obs:estimate_points_on_straight_line}
	Let $\Delta$ be a straight line segment, $|\Delta| = b$ and $M \in \mathfrak{M}(2,\mathbb{N})$.
	Let $\#(\Delta \cap M) = k$ and $b>10000$.
	Then $k \leq \gamma_2 b + 6$,
	where
	\begin{equation}
		\gamma_2 = \frac{3846}{2593 \sqrt{647}-5823} = 0.063958...
	\end{equation}
\end{proposition}

\begin{proof}
	We know the maximum number of points for planar integral point sets of diameters at most $10000$.
	Thus, we are interested in integral $t$ such that estimate~\eqref{eq:estimate_for_segment_length}
	holds for $b > 10000$.
	So, let us consider $t\geq 647$ and find a coefficient $\gamma_2$,
	such that the inequality
	\begin{equation}
		\label{eq:inequality_for_linear_estimate_of_segment_length}
		\frac{2}{3}t^{3/2}-\frac{3}{2}t+\frac{5 }{6}t^{1/2} \geq \frac{t-6}{\gamma_2}
	\end{equation}
	holds for all $t\geq 647$.
	For such $t$,
	the left-hand side of~\eqref{eq:inequality_for_linear_estimate_of_segment_length} obviously grows faster than the right-hand side.
	Turning~\eqref{eq:inequality_for_linear_estimate_of_segment_length} into the same equality
	and solving it for $t=647$, we obtain the required estimate.
\end{proof}

Now we are ready to prove the main theorem of the section.

\begin{theorem}
	\label{thm:main_estimate}
	If $n\geq 4$, then $d(2,n) \geq \gamma (n - 2)$,
	where
	\begin{multline}
		0.46530... =
		3^{1/4} \cdot 2^{-3/2} >
		\\ >
		\gamma = \frac{\sqrt{16 {{\left( \sqrt{\frac{2}{\sqrt{3}}+\frac{1}{21490}}+\frac{1}{\sqrt{21490}}\right) }^{2}}+\frac{14791716}{{{\left( 2593 \sqrt{647}-5823\right) }^{2}}}}-\frac{3846}{2593 \sqrt{647}-5823}}{8 {{\left( \sqrt{\frac{2}{\sqrt{3}}+\frac{1}{21490}}+\frac{1}{\sqrt{21490}}\right) }^{2}}}
		>\\
		> 0.45557
		> \frac{5}{11}
		.
	\end{multline}
\end{theorem}

\begin{proof}
	For $4 \leq n \leq 21491$, the assertion of the theorem follows immediately from Proposition~\ref{obs:4_leq_n_leq_21491}.
	Let us consider $M\in \mathfrak{M}(2,n)$, $n \geq 21492$, $\operatorname{diam} M = b$.
	Lemma~\ref{lem:square_container} yields that $M$ is situated in a square of side length $b$.
	Let $M_1, M_2, M_3, M_4$ be points of $M$ such that the distances $|M_1 M_2|$ and $|M_3 M_4|$
	are minimal in $M$ ($M_2$ and $M_3$ may coincide).
	Then $|M_1 M_2| \leq b\varphi_{n-1}$ and $|M_3 M_4| \leq b\varphi_{n-1}$.
	Due to Lemma~\ref{lemma:quadr_diag_edges}, open line segments $M_1 M_2$ and $M_3 M_4$ do not intersect
	(otherwise they are not minimal).

	Let $C = cr(M_1 M_2) \cap cr(M_3 M_4)$.
	Each point $N\in M$ satisfies one of the following conditions:

	a) $N$ belongs to $C$~--- overall at most $\gamma_2 b + 6$ points by Lemma~\ref{lem:intervals_cross} and Proposition~\ref{obs:estimate_points_on_straight_line};

	b) $N$ belongs to the intersection of one of $|M_1 M_2| - 1$ hyperbolas
	with one of $|M_3 M_4| - 1$ hyperbolas~--- overall at most $4 (|M_1 M_2| - 1)(|M_3 M_4| - 1)$ points;

	c) $N$ belongs to the intersection of one of $|M_1 M_2| - 1$ hyperbolas with $cr(M_3 M_4)$~---
	overall at most $4 (|M_1 M_2| - 1)$ points;

	d) $N$ belongs to the intersection of one of $|M_3 M_4| - 1$ hyperbolas with $cr(M_1 M_2)$~---
	overall at most $4 (|M_3 M_4| - 1)$ points;

	(see~\cite{erdos1945integral} for details).
	Summing up the above cases, we obtain the following estimate:
	\begin{equation}
		\label{eq:estimate_n_varphi_n_squared}
		n \leq 4 b^2 \varphi_{n-1}^2 - 4 + \gamma_2 b + 6
		.
	\end{equation}
	Proposition~\ref{varphi_n_where_n_geq_21492} turns estimate~\eqref{eq:estimate_n_varphi_n_squared}
	into the following:
	\begin{equation}
		\label{eq:estimate_n_sqrt_n-2}
		n-2 \leq 4 b^2 \frac{\beta^2}{n-2} + \gamma_2 b
		,
	\end{equation}
	which obviously leads to inequality
	\begin{equation}
		1 \leq 4\beta^2 \left(\frac{b}{n-2}\right)^2  + \gamma_2 \frac{b}{n-2}
		.
	\end{equation}
	Let us denote $\lambda = b/(n-2)$ and solve the following quadratic inequality for $\lambda$:
	\begin{equation}
		\label{eq:square_inequality_lambda}
		4\beta^2 \lambda^2  + \gamma_2 \lambda - 1 \geq 0
		.
	\end{equation}
	The discriminant is $\gamma_2^2 + 16 \cdot \beta^2$,
	so we obtain the following estimate:
	\begin{equation}
		\frac{b}{n-2} = \lambda \geq \frac{-\gamma_2 + \sqrt{\gamma_2^2 + 16 \cdot \beta^2}}{8\beta^2}
		.
	\end{equation}
	Calculating the expression above for our $\beta$ and $\gamma_2$, we conclude the proof.
\end{proof}

\begin{corollary}
	If $n\geq 4$, then $d(2,n) > \frac{5}{11} n$.
\end{corollary}

\begin{proof}
	For $4 \leq n \leq 21491$, the claim follows from Proposition~\ref{obs:4_leq_n_leq_21491} immediately.
	For $n > 21491$, the inequality $0.45557(n-2) > \frac{5}{11}n$ holds.
\end{proof}

We can improve the constant in Theorem~\ref{thm:main_estimate},
if some new exact values of $d(2,n)$ are found;
however, the obtained constant is bounded by $3^{1/4}\cdot2^{-3/2}$ above.
More precisely, we have the following theorem.

\begin{theorem}
	For every $\varepsilon > 0$ there exists a number $n_0$ such that the inequality
	\begin{equation}
		d(2,n) \geq n\cdot(3^{1/4}\cdot 2^{-3/2} - \varepsilon)
	\end{equation}
	holds for every $n>n_0$,
	that is
	\begin{equation}
		\lim_{n\to\infty} \frac{d(2,n)}{n} \geq 3^{1/4}\cdot 2^{-3/2}
		.
	\end{equation}
\end{theorem}

\section{Constructing integral point sets}

\begin{definition}
	A set $M\in\mathfrak{M}(2,n)$ is called \textit{facher}
	if $M$ consists of $n-1$ points on a straight line
	and one point out of the line.
\end{definition}

\begin{definition}
	A set $M\in\mathfrak{M}(m,n)$ is called \textit{optimal}
	if $\operatorname{diam}M=d(m,n)$.
\end{definition}

\begin{definition}
	\cite{kurz2005characteristic}
	A squarefree number $q$ is called the \textit{characteristic} of $M\in\mathfrak{M}(2,\mathbb{N})$,
	if for any points $M_1, M_2, M_3 \in M$ the area of triangle $M_1 M_2 M_3$
	is $p_{1,2,3}\sqrt{q}$ for some rational $p_{1,2,3}$.
\end{definition}
For a given $M\in\mathfrak{M}(2,\mathbb{N})$, the characteristic is determined uniquely.

Facher sets are the simplest planar integral point sets.
It is known that for $9\leq n \leq 122$ all the optimal sets are facher~\cite{kurz2008minimum}.
For every cardinality $n$ and every squarefree number $q$
there exists a facher set $M\in\mathfrak{M}(2,n)$ with characteristic $q$~\cite[Theorem 5]{our-vmmsh-2018}.
In~\cite{antonov2008maximal}, the facher sets of characteristic 1 were investigated; they were called \textit{semi-crabs}.

For every integer $n\geq 3$ Solymosi presented~\cite{solymosi2003note} a construction of a facher integral point set
$M\in\mathfrak{M}(2,n)$
such that equality $|M_1 M_2| = 2$ holds for some $M_1, M_2 \in M$.
The constructed set has both odd and even distances.

Now we improve Solymosi's result.

\begin{construction}
	\label{con:planar_set_with_minimeter_1}
	Let us choose a positive integer $k > 1$ and set
	\begin{equation}
		a = 2^{2^k} - 1
		.
	\end{equation}
	Then
	\begin{equation}
		\label{eq:a_equiv_3_mod_4}
		a \equiv 3 \mod 4
	\end{equation}
	and, moreover,
	\begin{multline}
		a = \left(2^{2^{k-1}}\right)^2 - 1
		=
		\left(2^{2^{k-1}} + 1\right) \left(2^{2^{k-1}} - 1\right)
		=
		\\=
		\left(2^{2^{k-1}} + 1\right) \left(2^{2^{k-2}} + 1\right) \cdot ... \cdot \left(2^{2^1} + 1\right) \left(2^2 - 1\right)
		.
	\end{multline}
	We set $d_j = 2^{2^j} + 1$ for $1 \leq j \leq k-1$.
	Then $d_j \equiv 1 \mod 4$.

	Let $c_J = \prod_{j\in J} d_j$ for every subset of indices $J\subset I = \{1,2,...,k-1\}$
	(and  $c_\varnothing = 1$).
	We obtain
	\begin{equation}
		\label{eq:c_J_equiv_1_mod_4}
		c_J\equiv 1 \mod{4}
	\end{equation}
	and, moreover, $a$ is divisible by $c_J$.

	Statements~\eqref{eq:a_equiv_3_mod_4} and~\eqref{eq:c_J_equiv_1_mod_4} yield $a/c_J \equiv 3 \mod 4$.
	Let $b_J = (c_J - a/c_J)/2$, then $b_J \equiv 1 \mod 2$.
	Let us further take $g_J = (c_J + a/c_J)/2$, then $g_J \equiv 0 \mod 2$.

	Next we define the coordinates of the points as following:
	\begin{equation}
		M_{J\pm} =\left(\pm\frac{b_J}{2}, 0\right)
		,
		~~
		N   =\left(0, \frac{\sqrt{a}}{2}\right)
		.
	\end{equation}
	Then the distances are:
	\begin{multline}
		|N M_{J\pm}|
		=
		\left(\frac{b_J^2}{4} + \frac{a}{4}\right)^{1/2}
		=
		\frac{1}{2}\left(\left(\frac{c_J - a/c_J}{2}\right)^2 + a\right)^{1/2}
		=
		\\=
		\frac{1}{2}\left( \left(\frac{c_J}{2}\right)^2 - \frac{a}{2} + \left(\frac{a}{c_j}\right)^2 + a\right)^{1/2}
		=
		\frac{1}{2}\left( \left(\frac{c_J}{2}\right)^2 + \frac{a}{2} + \left(\frac{a}{c_j}\right)^2    \right)^{1/2}
		=
		\\=
		\frac{1}{2}\left(\left(\frac{c_J + a/c_J}{2}\right)^2\right)^{1/2}
		=
		\frac{1}{2}\left(\frac{c_J + a/c_J}{2}\right)
		=
		\frac{g_J}{2}
		\in\mathbb{N}
		,
	\end{multline}
	\begin{equation}
		|M_{J_1 \pm}  M_{J_2 \pm}|
		=
		\left|\frac{b_{J_1}}{2} \pm \frac{b_{J_2}}{2}\right|
		=
		\left|\frac{b_{J_1} \pm b_{J_2}}{2}\right|
		\in\mathbb{N}
		.
	\end{equation}
	In particular, for $H = \{k-1\}$ we obtain $C_H = 2^{2^{k-1}}+1$ and
	\begin{equation}
		b_H =
		\left( 2^{2^{k-1}}+1 - \frac{a}{2^{2^{k-1}}+1} \right)/2
		=
		\left(2^{2^{k-1}}+1 - \left(  2^{2^{k-1}}-1 \right) \right)/2
		=
		1
		.
	\end{equation}
	Thus, one of the distances is
	\begin{equation}
		|M_{H+}  M_{H-}|
		=
		\left|\frac{b_{H}}{2} - \frac{-b_{H}}{2}\right|
		=
		\left|\frac{1}{2} - \frac{-1}{2}\right|
		= 1
		.
	\end{equation}

	Note that all the points $M_{J\pm}$ are distinct:
	the equality $b_J =  b_K$ implies $J=K$;
	the equality $b_J = -b_K$ implies $c_J = -c_K$ or $c_J = a / c_K$.
	The first case is impossible because both $c_J$ and $c_K$ are positive;
	the second case contradicts~\eqref{eq:a_equiv_3_mod_4} and~\eqref{eq:c_J_equiv_1_mod_4}.

	So, $M = \{ M_{J\pm}, N\}$ is indeed a planar integral point set of $2^k+1$ points,
	and distance 1 occurs in $M$.
	Since $k$ can be taken arbitrary large, it follows that we can construct
	a planar integral point set of arbitrary large cardinality so that the distance 1 occurs in that set.
\end{construction}

\begin{remark}
	\label{rem:planar_set_with_minimeter_1_of_3_points}
	Applying Construction~\ref{con:planar_set_with_minimeter_1} to $k=1$ naively,
	we get $a = 3$, $I=\varnothing$, $b_\varnothing = -1$ and then
	obtain an equilateral triangle of side length 1.
	This triangle is obviously the optimal set in $\mathfrak{M}(2,3)$.
\end{remark}

\begin{remark}
	\label{rem:planar_set_with_minimeter_1_of_4_and_5_points}
	For $k=2$ in Construction~\ref{con:planar_set_with_minimeter_1},
	we obtain the optimal set in $\mathfrak{M}(2,5)$ presented in~\cite[Fig. 1]{harborth1993upper}.
	If we remove one point from it,
	then we get one of the two optimal sets in $\mathfrak{M}(2,4)$.
\end{remark}

\begin{definition}
	\cite{antonov2008maximal}
	A set $M\in \mathfrak{M}(m,n)$ is \textit{maximal},
	if there is no set $M'\in \mathfrak{M}(m,n+1)$
	such that $M \subsetneq M'$.
\end{definition}


Setting $k=1$ in Lemma~\ref{lem:2k-1_segments},
we obtain the following result (which is exactly~\cite[Lemma 3]{our-vmmsh-2018}).
\begin{corollary}
	\label{cor:only_one_distance_1_on_straight_line}
	Let $M \in \mathfrak{M}(2,\mathbb{N})$ and let $m$ be a straight line.
	Then there is at most one pair of points $M_1,M_2\in M \cap m$
	such that $|M_1 M_2| = 1$.
\end{corollary}

In order to describe all the sets $M \in \mathfrak{M}(2,n)$ with distance 1,
we need~\cite[Proposition 6]{our-vmmsh-2018}.
For the reader's convenience, we will state a slightly rephrased version of it and provide the proof.
\begin{lemma}
	\label{lem:no_4_points_in_semigeneral_position_with_distance_1}
	Let $M=\{M_1,M_2,M_3,M_4\}\in \mathfrak{M}(2,4)$ such that $|M_1 M_2|=1$.
	If $l$ is the straight line through points $M_1$ and $M_2$,
	then one of the points $M_3$ or $M_4$ belongs to $l$.
\end{lemma}

\begin{proof}
	Let $m$ denote the perpendicular bisector of line segment $M_1 M_2$.
	Due to the triangle inequality, there are two possibilities to place each of points $M_i$, $i=3,4$:

	a) $M_i$ belongs to $l$ and $|M_i M_1|-|M_i M_2| = \pm 1$;

	b) $M_i$ belongs to $m$ and $|M_i M_1| = |M_i M_2|$.

	If (a) is true for both points, then $M\subset l$ and thus $M$ is not an integral point set.
	If (a) is true for one point and (b) is true for another, then the claim of the lemma follows.
	So, suppose to the contrary, that both points $M_3$ and $M_4$ belong to $m$.

	The area of the triangle $M_1 M_3 M_4$ is rational because $|M_3 - M_4| \in \mathbb{Z}$.
	Thus, the characteristic of $M$ is 1,
	and there is a Cartesian coordinate system such that $M_1=(-1/2,0)$, $M_2=(1/2,0)$, $M_3=(0, a/2)$
	(see~\cite[Theorem 4]{our-vmmsh-2018}).
	It is clear that $a\neq 0$.
	We set $b = |M_1 - M_3|$ and $O=(0,0)$.
	Applying the Pythagorean theorem to triangle $OM_1M_3$, we obtain the following Diophantine equation:
	\begin{equation}
		\frac{1}{4} + \frac{a^2}{4} = b^2
		,
	\end{equation}
	or, equivalently,
	\begin{equation}
		1 + a^2 = (2b)^2
		.
	\end{equation}
	This equation has no integral solutions.
	This contradiction concludes the proof.
\end{proof}

Using Construction~\ref{con:planar_set_with_minimeter_1},
Lemma~\ref{lem:no_4_points_in_semigeneral_position_with_distance_1} and Corollary~\ref{cor:only_one_distance_1_on_straight_line}
together with results of~\cite[Section 6]{antonov2008maximal},
we obtain the following theorem.

\begin{theorem}
	\label{thm:minimeter_1_planar}
	For every $n\geq 3$ there is a planar integral set $M$ of $n$ points
	such that for some $M_1,M_2 \in M$ equality $|M_1 M_2|=1$ holds.
	This set consists of $n-1$ points, including $M_1$ and $M_2$, on a straight line and one point out of the line.

	And vice versa, if $M$ is a planar integral point set of $n$ points
	such that for some $M_1,M_2 \in M$ equality $|M_1 M_2|=1$ holds,
	then $M$ consists of $n-1$ points, including $M_1$ and $M_2$, on a straight line,
	and one point out of the line, on the perpendicular bisector of line segment $M_1 M_2$.
	There is only one maximal integral point set $M' \supseteq M$,
	and the perpendicular bisector is the axis of symmetry for $M'$.
	Moreover, if $n > 3$, then the distance 1 occurs in $M$ (and $M'$) only once.
\end{theorem}

Using the ``blowing up'' procedure described in~\cite[theorem 1.3]{kurz2008bounds},
we can construct an integral point set $M$ in $m$-dimensional Euclidean space, $m\geq 3$,
with distance 1 occuring in it.
If we have an integral point set with distance 1 occuring,
then we can easily dilate it to turn 1 into the desired distance.
These facts, together with Theorem~\ref{thm:minimeter_1_planar}, give the following theorem.

\begin{theorem}
	For arbitrary integers $m \geq 2$, $n \geq m+1$, $d \geq 1$
	there exists $M\in\mathfrak{M}(m,n)$,
	such that for some $M_1, M_2\in M$ equality $|M_1 M_2| = d$ holds.
\end{theorem}

\begin{definition}
	\cite{noll1989nclusters}
	We will call an integral point set $M$ \textit{prime}, if the greatest common divisor
	of all the distances occuring in $M$ is 1.
\end{definition}

If an integral point set is prime,
then it cannot be squashed to an integral point set of the same power and structure but smaller diameter.

\begin{theorem}
	\label{thm:prime_with_desired_distance}
	For every $m \geq 3$, $n \geq m+1$, $d \geq 1$ there exists a \textbf{prime}
	integral point set $M\in\mathfrak{M}(m,n)$ which
	contains points $M_1$ and $M_2$ such that distance between $M_1$ and $M_2$ is exactly $d$.
\end{theorem}

\begin{proof}
	Take a facher integral point set according to Construction~\ref{con:planar_set_with_minimeter_1}
	with sufficiently large height and then apply the ``blowing up'' procedure described in~\cite[Theorem 1.3]{kurz2008bounds},
	replacing the out-of-line point with an $(m-2)$-dimensional simplex of proper side length.
	That simplex will consist of $m-1$ points;
	we can throw out up to all points $M_J$, except the two points with distance 1.
	So, the set consists of $m+1$ or more points and has distance 1 occuring in it.
	As distance 1 occurs in the obtained set, it is prime.
\end{proof}

A slightly stronger result can also be claimed.

\begin{theorem}
	\label{thm:prime_with_unique_desired_distance}
	For arbitrary integers $m > 2$, $n \geq m+1$, $d \geq 1$
	there exists a prime set $M\in\mathfrak{M}(m,n)$,
	such that for some $M_1, M_2\in M$ equality $|M_1 M_2| = d$ holds,
	the distance $d$ is minimal and occurs in $M$ only once.
\end{theorem}

\begin{proof}
	We apply Construction~\ref{con:planar_set_with_minimeter_1} to obtain $M'\in \mathfrak{M}(2,n-m+2)$
	with $M'_1, M'_2 \in M'$ and $|M'_1 M'_2| = 1$.
	Dilate $M'$ to $M''$ in such a way that $M'_1$ and $M'_2$ turn into $M''_1$ and $M''_2$ resp.
	and $|M''_1 M''_2| = d$.
	Then ``blow up'' $M''$ to $M\in \mathfrak{M}(m,n)$ using a simplex of side length $d+1$.
	The greatest common divisor of $d$ and $d+1$ is $1$, so $M$ is prime.
\end{proof}

\section{Final remarks and open problems}

\begin{remark}
	The bound of Theorem~\ref{thm:main_estimate} is not tight.
	The derivation of tight bounds for the minimum diameter $d(2, n)$
	is still a challenging task for the forthcoming investigation~\cite[Section 7]{kurz2008minimum}.
\end{remark}
The known upper bound is $d(2,n)\leq 2^{c \log n \log \log n}$~\cite{harborth1993upper}.
However, we hope that the present article can provide a framework for possible better estimates.

\begin{conjecture}
	The approach of Theorem~\ref{thm:main_estimate} can be generalized to higher dimensions.
\end{conjecture}
For the overview of current lower bounds for higher dimensions, we refer the reader to~\cite{nozaki2013lower}.

There are also several ways to improve the current bound,
based on Lemma~\ref{lem:square_container}.

\begin{problem}
	\label{prb:square_side_length_and_diameter_power}
	Is there a number $\delta < 1$ such that any $M\in\mathfrak{M}(2,\mathbb{N})$ with $\operatorname{diam} M = d$
	is situated in a square of side length $d^\delta$?
\end{problem}
If the answer for Problem~\ref{prb:square_side_length_and_diameter_power} is affirmative,
then we can obtain a bound better than linear.

\begin{problem}
	\label{prb:square_side_length_and_diameter_linear}
	What are the maximal numbers $\delta \leq 1$ and $h \leq 0$
	such that any $M\in\mathfrak{M}(2,\mathbb{N})$ with $\operatorname{diam} M = d$
	is situated in a square of side length $\delta d - h$?
\end{problem}

If the answer for Problem~\ref{prb:square_side_length_and_diameter_linear} is not $\delta = 1$ and $h = 0$,
then the linear bound can be slightly improved.

However, Problems~\ref{prb:square_side_length_and_diameter_power} and~\ref{prb:square_side_length_and_diameter_linear}
do not seem to show the way for improving the bound
due to the following well-known theorem,
which can be proved by an application of inversion~\cite{solymosi2010question}.
\begin{theorem}
	\label{thm:rational_set_on_circle}
	There exists a dense subset $P$ of the unit circle such that for any $P_1, P_2 \in P$
	the distance $|P_1 P_2|$ is rational.
\end{theorem}
We conjecture that thorough investigation of existing examples of circular integral point sets
~\cite{bat2018number,harborth1993upper,piepmeyer1996maximum}
will lead to the negative answers for the
Problems~\ref{prb:square_side_length_and_diameter_power} and~\ref{prb:square_side_length_and_diameter_linear}.

Another possibility to increase the constant in our lower bound is introduced by the following problem.

\begin{problem}
	\label{prb:minimal_shape}
	What is the minimal shape $\mathcal{S}$ such that any $M\in\mathfrak{M}(2,\mathbb{N})$ with $\operatorname{diam} M = d$
	is situated in  $\mathcal{S}$?
\end{problem}
This problem is another generalization of Lemma~\ref{lem:square_container}.
We know that $\mathcal{S}$ is not a circle of diameter $d$
(the counterexample is an equilateral triangle of side length $d$)
and we conjecture that $\mathcal{S}$ is not a Reuleaux triangle of width $d$
(the possible counterexample belongs to a circle of diameter $d$, see Theorem~\ref{thm:rational_set_on_circle}).

The solution to Problem~\ref{prb:minimal_shape} (if ever found) will lead to a sophisticated packing problem,
something like ``How to put $n$ points in a convex hull of concentric circle and Reuleaux triangle?''
Obviously, Problem~\ref{prb:minimal_shape} can be generalized to higher dimensions.

For $n=3,4,5$, there are optimal sets in $\mathfrak{M}(2,n)$ that contain distance 1
(see Remarks~\ref{rem:planar_set_with_minimeter_1_of_3_points} and~\ref{rem:planar_set_with_minimeter_1_of_4_and_5_points}).
However, we can suggest the following conjecture.
\begin{conjecture}
	\label{con:no_optimal_with_edge_1}
	Every set $M\in\mathfrak{M}(2,n)$, $n\geq 6$, such that for some $M_1,M_2$ equality $|M_1 M_2|=1$ holds,
	is not optimal.
\end{conjecture}

The motivation for Conjecture~\ref{con:no_optimal_with_edge_1} is based on the results of~\cite[Section 5]{kurz2008minimum},
especially the following theorem.

\begin{theorem}
	For every $12 \leq n \leq 122$, there exist a facher optimal sets $M_n\in\mathfrak{M}(2,n)$
	that consist of $n-1$ points $\{ (b_1,0), ..., (b_{n-1},0)\}$
	and the point $(0,\sqrt{A_n})$,
	where $b_1,...,b_{n-1}, A_n$ are integer and $\sqrt{A_n}\notin\mathbb{N}$.
\end{theorem}
Such a set can not contain distance 1: due to Theorem~\ref{thm:minimeter_1_planar},
the points with distance 1 should be $(\pm 1/2, 0)$.

For planar integral point sets containing distance 1, we can emphasise on the following problem.
\begin{problem}
	For given $n\geq 3$ and squarefree number $q \neq 1$,
	is there $M\in\mathfrak{M}(2,n)$, such that characteristic of $M$ is $q$
	and distance 1 occurs in $M$?
\end{problem}

The consideration of integral point sets containing distance 1 in higher dimensions gives a rise to another problem.

\begin{problem}
	How many times can the distance 1 occur in a set $M\in\mathfrak{M}(m,n)$?
\end{problem}
The answer for $m=2$ is given by Theorem~\ref{thm:minimeter_1_planar}.
For greater $m$, we will be bold enough to suggest the following conjecture.
\begin{conjecture}
	Consider $M\in\mathfrak{M}(m,n)$, $M_1,M_2\in M$, $|M_1 M_2| =1$.
	Then either $n = m+1$ or $M$ is a ``blowup'' of a facher set and distance 1 occurs in $M$
	no more than $1+\frac{(m-1)^2 - (m-1)}{2}$ times.
\end{conjecture}
Such estimate is based on the number of edges in an $(m-2)$-dimensional simplex.

We want to conclude the list of conjectures with the one which appears the easiest to deal with.

\begin{conjecture}
	\label{hyp:prime_planar}
	For every $n \geq 3$, $d \geq 1$ there exists an \textbf{prime}
	set $M\in\mathfrak{M}(2,n)$ which
	contains points $M_1$ and $M_2$ such that distance between $M_1$ and $M_2$ is exactly $d$.
\end{conjecture}
This conjecture gives the same claim as Theorem~\ref{thm:prime_with_desired_distance} does, but for $m=2$.
Obviously, Theorem~\ref{thm:prime_with_unique_desired_distance} can be generalized (for $m=2$) into a similar conjecture.

\section{Acknowledgements}
Author thanks Dr. Prof. E.M. Semenov for the fruitful discussion and ideas,
A.S. Chervinskaia for the idea of using the term ``facher'' and other linguistic advices,
and Dr. A.S. Usachev for proofreading.

\printbibliography

\end{document}